\documentclass[12pt,oneside]{article}
\usepackage[all]{xy}
\usepackage{amsmath,amssymb,amsfonts,amsthm,mathrsfs,dsfont,mathtools}
\usepackage[english]{babel}
\usepackage{graphicx,stackrel}
\usepackage{multirow,rotating,paralist}
\usepackage[toc,page,header]{appendix}
\DeclareMathAlphabet{\mathpzc}{OT1}{pzc}{m}{it}
\usepackage{minitoc}
\usepackage{extarrows}
\usepackage{hyperref}
\usepackage{csquotes}
 \small {\par\vskip8pt minus3pt\rm}
\newcounter{item}[section]
\newcounter{kirshr}
\newcounter{kirsha}
\newcounter{kirshb}

\newtheorem{theorem}{Theorem}[section]

\newcommand\undersym[2]{\raisebox{-6pt}{\tiny$#2$}{\kern-5pt}\mbox{$#1$}}
\newcommand\overcirc[1]{\raisebox{10pt}{\tiny$\circ$}{\kern-7pt}\mbox{$#1$}}
\title{\Large{Cyclic Automorphisms groups of genus $10$ non-hyperelliptic curves}}
\date{}
\author{{\small Eslam E. Badr$^{1}$ and Mohammed A. Saleem$^{2}$}}

\begin{document}
\bibliographystyle{plain}
\maketitle                     % Produces the title.
\vspace*{-.95cm}

\begin{center}
${^1}$Department of Mathematics, Faculty of Science, \\[-0.15cm]
Cairo University, Giza, Egypt \\[-0.15cm]
\vspace*{.3cm}
${^2}$Department of Mathematics, Faculty of Science,\\[-0.15cm]
Sohag University, Sohag, Egypt \\[-0.15cm]
\end{center}

\begin{center}
$
\begin{array}{ll}
\text{Emails:}&\text{eslam@sci.cu.edu.eg} \\
  &\text{abuelhassan@science.sohag.edu.eg}
\end{array}
$
\end{center}

\maketitle \bigskip
%%%%%%%%%%%%%%%%%%%%%%%%%%%%%%%%%%%%%%%%%%%%%
\begin{abstract}
In this paper, we list all cyclic automorphisms subgroups $H$ for which there exists a smooth projective non-hyperelliptic sextic curve $C$ with $H\preceq Aut(C)$. Furthermore, we attach to each group a defining equation of a plane sextic curve having exactly this group as cyclic automorphisms subgroup. These equations cover up to isomorphism all plane non-singular sextic curves having some non-trivial automorphism.
\end{abstract}

\newpage
%%%%%%%%%%%%%%%%%%%%%%%%%%%%%%%%%%%%%%%%%%%
%%%%%%%%%%%%%%%%%%%%%%%%%%%%%%%%%%%%%%%%%%%
%%%%%%%%%%%%%%%%%%%%%%%%%%%%%%%%%%%%%%%%%%
\section{Introduction}
\par At the beginning of the 19th century, a vast literature on automorphism groups of algebraic curves has been found by
Schwartz, Klein \cite{C4}, Hurwitz \cite{H4}, Wiman \cite{W1}, \cite{W2} and others, we refer for example to \cite{pa23}. However, most of the literature is quite recent.
\par By covering space theory, a finite group $G$ acts (faithfully) on a genus $g$ curve if and only if it has a genus $g$
generating system. Using this purely group-theoretic condition, Breuer \cite{C2} used the computer algebra system GAP \cite{C3}
with major computational effort and classified all groups that act on a curve of genus $\leq48$. It greatly improved on several papers dealing with small genus, by various authors.
\par In \cite {M1}, Magaard and other authors studied the locus of genus $g\geq2$ curves that admit a finite group $G$-action of given type. Furthermore, they studied the locus of genus $g$ curves with prescribed automorphism group $G$. In other words, they classified these loci for $g=3$ (including equations for the corresponding curves), and for $g\leq10$ they classified those loci corresponding to large $G$.
\par In \cite{pa20}, \cite{B2}, the present authors used automorphism groups\textquoteright\, actions to classify and investigate the geometry of $1$-Weierstrass points on certain families of genus $3$ curves, called \emph{Kuribayashi quartic curves.}
\par In \cite{B1}, we computed the $1$-gap sequences of $1$-Weierstrass points of non-hyperelliptic smooth projective curves of genus $10$. Furthermore, the geometry of such points is classified as flexes, sextactic and tentactic points.
Also, an upper bounds for their numbers are estimated.
\par The present paper is organized in the following manner. In section $1$, we survey some of known results about the $Aut(X_g)$, where $X_g$ is a compact Riemann surface of genus $g.$ In section $2,$ we establish our main results Theorems \ref{thm20} that concern with the list of cyclic automorphisms subgroups $H$ for which there exists a smooth projective non-hyperelliptic sextic curve $C$ with $H\preceq Aut(C)$. Finally, we conclude the paper with some  remarks, comments and related problems.

 %%%%%%%%%%%%%%%%%%%%%%%%%%%%%%%%%%%%%%%%%%%
%%%%%%%%%%%%%%%%%%%%%%%%%%%%%%%%%%%%%%%%%%%
%%%%%%%%%%%%%%%%%%%%%%%%%%%%%%%%%%%%%%%%%%
\section{Preliminaries} In this section we mention some well known facts about the full automorphisms group of a compact Riemann surface $X_g$. For more details, we refer, for example, to \cite{pa23}, \cite{M1}.
\par It was first proved by Schwartz that $Aut(X)$ is a finite group. Moreover, the group $Aut(X_g)$ acts on the set $WP(X_g)$
of Weierstrass points of $X_g$. This action is faithful unless $X_g$ is hyperelliptic, in which case its kernel is the group of
order $2$ containing the hyperelliptic involution of $X_g$.
\par Let $K$ be the projective quartic curve defined by \[X^3Y+Y^3Z+Z^3X=0,\]
which is now known as Klein\textquoteright s quartic. F. Klein \cite{C4} found all its automorphisms which is the simple group of order 168.
\par  Fuchsian groups are basic for the analytic theory of Riemann surfaces, they have been used heavily in the study of $Aut(X_g)$. They have been introduced by Poincar\'{e} in 1882, for more details, we refer for example to \cite{C2}, \cite{Ri}.
\par The next milestone was Hurwitz\textquoteright s seminal paper \cite{H4} in 1893, where he discovered an upper bound on $Aut(X_g)$, which is known by Riemann-Hurwitz formula. From which he derived that \[|\,Aut(X_g)\,|\leq\,84(g-1).\]
Curves which attain this bound are called Hurwitz curves and their automorphism groups Hurwitz groups. For example the
Klein\textquoteright s quartic is a Hurwitz curve.
\par Let $N(g)$ denotes the maximum of the $|Aut(X_g)|$ for a fixed $g\geq2.$ Accola \cite{Ac1} and Maclachlan \cite{Mc1}
independently show that $N(g)\geq8(g+1)$ and this bound is sharp for infinitely many $g$\textquoteright s. Moreover, $N(g)\geq8(g+3)$ if $g$ is divisible by $3$.
\par Let $m$ be the order of an automorphism of $X_g$. Hurwitz \cite{H4} showed that \[m\leq10(g-1).\] In 1895, Wiman improved this bound to be \[m\leq2(2g+1)\] and showed this is best possible. If $m$ is a prime then $m\leq2g+1$. Homma \cite{Ho} shows that this bound is achieved if and only if the curve is birationally equivalent to
\[y^{m-s}(y-1)^s=x^q\quad for \,\,\,\,1\leq s\leq m\leq g+1.\]
\par A subgroup $G$ of $Aut(X_g)$ is said to be a large automorphism group in genus $g$ if
\[|\,G\,|>4(g-1).\]
So that, the quotient of $X_g$ by $G$ is a curve of genus $0$, and the number of points of this quotient ramified in $X_g$ is $3$
or $4$ (see \cite{C2}, \cite{FK}). Singerman \cite{Si3} shows that Riemann surfaces with large cyclic, Abelian, or Hurwitz groups are symmetric (admit an involution). Kulkarni \cite{Ku} classifies Riemann surfaces admitting large cyclic automorphism groups and works out full automorphism groups of these surfaces. Matsuno \cite{Mt} investigates the Galois covering of the projective line from compact Riemann surfaces with large automorphism groups.

\par Finally, it should be mentioned that algebraic curves of genus $10$ can be embedded holomorphically and effectively on algebraic curves of degree $6$.

%%%%%%%%%%%%%%%%%%%%%%%%%%%%%%%%%%%%%%%%%%%%%%%%%%%%%
%%%%%%%%%%%%%%%%%%%%%%%%%%%%%%%%%%%%%%%%%%%%%%%%%%%%%
%%%%%%%%%%%%%%%%%%%%%%%%%%%%%%%%%%%%%%%%%%%%%%%%%%%%%
\section{Main theorem}
\par The idea to obtain the results is to use a cyclic subgroup $H$ of order $m$, $H\preceq Aut(C)$, in order to obtain a model equation for $C$. It should be noted that we apply the same process used by Dolgachev in \cite{C1}.
\begin{theorem}\label{thm20}
Let $\varphi$ be a non-trivial automorphism of order $m$ of a smooth projective non-hyperellptic sextic curve $C: F(X;Y;Z)=0.$ Let us choose coordinates such that the generator of the cyclic group $H =<\varphi>$ is represented by the diagonal matrix
\[\varphi:=\left(
             \begin{array}{ccc}
               1 & 0 & 0 \\
               0 & \xi_m^a & 0 \\
               0 & 0 & \xi_m^b \\
             \end{array}
           \right),
\]
where,  $\xi_m$ is a primitive $m$-th root of unity. Then $F(X;Y;Z)$ is in the following list:
\begin{center}
  \textbf{Cyclic automorphism of order $m$}\\
\end{center}
\begin{tabular}{|c|c|}
  \hline
  Type: $m, (a,b)$ & $F(X;Y;Z)$ \\\hline\hline
   $30,(5,6)$ & $X^6+Y^6+Z^5X$ \\\hline
   $25,(2,15)$& $X^6+\alpha Z^2X^4+\beta Z^3X^3+\gamma Z^4X^2+Z^5X+Y^5Z$ \\\hline
   $24,(1,19)$& $X^6+Y^5Z+YZ^5$ \\\hline
   $21,(4,5)$&  $X^5Y+Y^5Z+Z^5X$\\\hline
   $15, (10,9)$& $X^6+Y^6+Z^5X+\alpha Z^3X^3$\\\hline
   $12, (7,1)$&  $X^6+Y^5Z+YZ^5+\alpha Y^3Z^3$\\\hline
   $10, (5,2)$&  $X^6+Y^6+Z^5X+\alpha X^4Y^2+\beta X^2Y^4$\\\hline
   $8, (1,3)$&  $X^6+Y^5Z+YZ^5+\beta X^2Y^2Z^2$\\\hline
   $6, (0,1)$&  $\beta Z^6+L_{6,Z}$\\\hline
   $6, (5,1)$&  $X^6+Y^6+Z^6+\alpha X^4YZ+\beta X^2Y^2Z^2+\gamma Y^3Z^3$\\\hline
   $6, (5,2)$&  $X^6+Y^6+Z^6+\alpha X^3Y^2Z+\beta Z^3X^3+\gamma Y^2Z^4+\delta Y^4Z^2$\\\hline
   $5, (4,3)$&  $X^6+\alpha X^3YZ^2+\beta X^2Y^3Z+XY^5+Z^5X+\delta Y^2Z^4$\\\hline
   $5, (0,1)$&  $Z^5L_{1,Z}+L_{6,Z}$\\\hline
   $4, (1,3)$&  $X^6+Y^5Z+YZ^5+\alpha X^4YZ+\beta Y^3Z^3+\gamma Z^4X^2+\delta X^2Y^2Z^2+\mu X^2Y^4$\\\hline
   $3, (0,1)$&  $\beta Z^6+Z^3L_{3,Z}+L_{6,Z}$\\\hline
   $3, (1,2)$&  $X^5Y+Y^5Z+Z^5X+\alpha Z^2X^4+\beta X^2Y^4+\gamma Y^2Z^4+\delta X^3Y^2Z+\mu XY^3Z^2+\eta X^2YZ^3$\\\hline
   $2, (0,1)$&  $\beta Z^6+Z^4L_{2,Z}+Z^2L_{4,Z}+L_{6,Z}$\\\hline
  \hline
\end{tabular}\\\\
where $L_{i,B}$ denotes a generic denotes a homogenous polynomial of degree $i$ with variables different from the variable $B$.
\end{theorem}

\begin{proof}
Let $C$ be a non-singular plane sextic curve (i.e. with degree$\geq$ in each variable) and let $\varphi$ act by $(X:Y:Z)\mapsto(X:\xi_m^aY:\xi_m^bZ).$
Now, we have two cases, either $ab=0$ or $ab\neq0$. In the following we treat each of these cases.\\ \par\textbf{Case I:} Suppose first that $ab=0$. One can assume without loss of generality that $a=0$, (otherwise with the change of variables $Y\leftrightarrow Z$ we should obtain the same results). Write:
\[F(X;Y;Z)=\lambda Z^6+Z^5L_{1,Z}+Z^4L_{2,Z}+Z^3L_{3,Z}+Z^2L_{4,Z}+ZL_{5,Z}+L_{6,Z}.\]
\par If $\lambda=0$, then $5b\equiv0(mod\,m)$, then $m=5.$ Thus, \[L_{2,Z}=L_{3,Z}=L_{4,Z}=L_{5,Z}=0,\] and we get Type $5, (0,1).$ \par If $\lambda\neq0,$ then $6b\equiv0(mod\, m)$, then $m=2,3$ or $6.$ If $m=2,$ then \[L_{1,Z}=L_{3,Z}=L_{5,Z}=0,\] and we get Type $2, (0,1).$ If $m=3,$ then \[L_{1,Z}=L_{2,Z}=L_{4,Z}=L_{5,Z}=0,\] and we get Type $3, (0,1).$ If $m=6,$ then
\[L_{1,Z}=L_{2,Z}=L_{3,Z}=L_{4,Z}=L_{5,Z}=0,\] and we get Type $6, (0,1).$\\
\par\textbf{Case II:} Suppose that $ab\neq0.$ We can suppose that $a\neq b$ and $g.c.d(a,b)=1$ (otherwise by scaling we could reduce to the first situation).
Then necessarily $m>2.$ Let \[P_1=(1;0;0),\quad P_2=(0;1;0),\quad and\,\,\,P_1=(0;0;1)\] be the reference points. Consequently, we have the following four subcases:
\begin{description}
\item[i.] All reference points lie in $C,$
\item[ii.] Two reference points lie in $C,$
\item[iii.] One reference points lie in $C,$
\item[iv.] None of the reference points lie in $C.$
\end{description}

\begin{itemize}
\item If all reference points lie in the smooth plane sextic curve $C$, then the possibilities for the defining equation are now:
\begin{eqnarray*}
C&:& X^5L_{1,X}+Y^5L_{1,Y}+Z^5L_{1,Z}+X^4L_{2,X}+Y^4L_{2,Y}+Z^4L_{2,Z}\\
&+&X^3L_{3,X}+Y^3L_{3,Y}+Z^3L_{3,Z}.
\end{eqnarray*}
It is obvious that $B_i$ cannot appear in $L_{1,B_j}$ whenever $j\neq i$, where \[B_1:=X,\, B_2:=Y\,\,\, \text{and}\,\,\, B_3:=Z.\] Hence,
by change of the variables $X,Y$ and $Z$, we can assume that
\begin{eqnarray*}
C&:& X^5Y+Y^5Z+Z^5X+X^4L_{2,X}+Y^4L_{2,Y}+Z^4L_{2,Z}\\
&+&X^3L_{3,X}+Y^3L_{3,Y}+Z^3L_{3,Z}.
\end{eqnarray*}
The first three factors implies that $a=5a+b=5b\,(mod\,m).$ Therefore, $m=3,7$ or $21.$
\par If $m=3$, then we can take a generator of $H$ such that $(a,b)=(1,2)$. By checking each monomial invariance, we obtain Type $3, (1,2).$
\par If $m=7,21$, then we can take a generator of $H$ such that $(a,b)=(1,3), (4,5)$ respectively. By checking each monomial\textquoteright s invariance, we obtain that no other monomial enters $C.$ Thus, we get Type $21, (4,5)$ in the table.
\item If two reference points lie in the smooth plane sextic curve $C$, then by re-scaling the matrix $\varphi$
and permuting the coordinates, we can assume that $(1;0;0)\notin C.$ The equation is then:
\[C: X^6+X^4L_{2,X}+X^3L_{3,X}+X^2L_{4,X}+XL_{5,X}+L_{6,X},\]
since $L_{1,X}$ is not invariant by $\varphi$ $(ab\neq0)$. Moreover, $Z^6$ and $Y^6$ are not in $L_{6,X}$
by assumption only $(1;0;0)\notin C.$\\
\par Assume first that $Y^5Z$ and $YZ^5$ are in $L_{6,X}.$ Then, $5a+b\equiv0\,(mod\,m)$ and $a+5b\equiv0\,(mod\,m)$, so that, $m|24.$
\par If $m=24$, then we can take a generator of $H$ such that $(a,b)=(1,19).$ Now, by checking each
monomial\textquoteright s invariance, we get that no other monomial enters $L_{6,X}.$ Moreover,
$L_{2,X}=L_{3,X}=L_{4,X}=L_{5,X}=0.$ Hence, we obtain Type $24, (1,19)$ in the table.
\par If $m=12$, then we can take a generator of $H$ such that $(a,b)=(7,1).$ Again, $L_{2,X}=L_{3,X}=L_{4,X}=L_{5,X}=0.$
But, the monomial $Y^3Z^3$ is invariant under $\varphi,$ consequently, we obtain type $12, (7,1).$
\par If $m=8$, then we can take a generator of $H$ such that $(a,b)=(1,3).$ Thus, $L_{2,X}=L_{3,X}=L_{5,X}=0$
and the monomial $X^2Y^2Z^2$ is invariant under $\varphi.$ Moreover no other monomial enters $C.$ Thus, we get Type $8,(1,3).$
\par If $m=4$, then we can take a generator of $H$ such that $(a,b)=(1,3).$ Thus, by checking each monomial\textquoteright s
invariance, we get that $L_{3,X}=L_{5,X}=0$ and $Y^3Z^3,\, X^2Y^4,\, Z^4X^2,\, X^2Y^2Z^2$ and $X^4YZ$ are invariant under $\varphi.$ Hence, we obtain Type $4, (1,3).$
\par Finally, the cases $m=3$ or $6$ can not happen under the condition that $a\neq b.$\\
\par Assume second that $Y^5Z\in L_{6,X}$ and $YZ^5\notin L_{6,X}.$ Then by non-singularity $Z^5$ is in $L_{5,X}$,
that is $5a+b\equiv0\, (mod\,m)$ and $5b\equiv0\, (mod\,m).$ Hence, $m=5$ or $25.$
\par If $m=25$, then we can take a generator of $H$ such that $(a,b)=(2,15).$ Thus, by checking each monomial\textquoteright s
invariance, we get that the monomials $X^4Z^2,\, X^3Y^3,\, Z^4X^2,$ and $Z^5X$ are invariant under $\varphi.$ So, we obtain Type
$25, (2,15)$ in the table. Moreover, the case $m=5$ can not happen under the assumption that $Y^5Z\in L_{6,X}$.\\
\par Up to a permutation of $Y\leftrightarrow Z$, we can assume that $Y^5Z$ and $YZ^5$ are not in $L_{6,X}.$ By non-singularity, we have that $Z^5$ and $Y^5$ should be in $L_{5,X}$, then $5b\equiv0\,(mod\,m)$ and $5a\equiv0\,(mod\,m).$
Therefore, $m=5$ and $(a,b)=(4,3)$, thus $L_{6,X}=0$ and the monomials $X^3YZ^2,\,X^2Y^3Z,\,Z^5X,\,XY^5,$ and $XY^2Z^4$ enter $C.$ So, we get Type $5, (4,3)$ in the table.
\item If one reference points lie in the smooth plane sextic curve $C$, then by normalizing the matrix $\varphi$ and
permuting the coordinates, we can assume that $(1;0;0),\,(0;1;0)\notin C.$ We can write
\[C: X^6+Y^6+X^4L_{2,X}+X^3L_{3,X}+X^2L_{4,X}+XL_{5,X}+L_{6,X},\]
such that $Z^6\notin L_{6,X}.$ Also, by non-singularity, we have that $Z^5\in L_{5,X},$ then $6a\equiv0\,(mod\,m)$ and
$5b\equiv0\,(mod\,m).$ So that, $m=10,15$ or $30.$
\par If $m=10,$ then we can take a generator of $H$ such that $(a,b)=(5,2).$ Thus, by checking each monomial\textquoteright s
invariance, we get that the monomials $X^4Z^2,\, X^3Y^3,\, Z^4X^2,$ and $Z^5X$ are invariant under $\varphi,$
moreover, $L_{3,X}=L_{6,X}=0.$ Hence, we obtain Type $10, (5,2).$
\par If $m=15,$ then we can take a generator of $H$ such that $(a,b)=(10,9).$ By checking each monomial\textquoteright s
invariance, we get that no other monomial enters $L_{5,X}.$ Moreover, $L_{2,X}=L_{4,X}=0$ and the monomial $Z^3X^3$
is invariant under $\varphi.$ Consequently, we obtain Type $15, (10,9)$ in the table.
\par If $m=30,$ then $(a,b)=(5,6).$ By checking each monomial\textquoteright s
invariance, we get that no other monomial enters $C.$ Thus, we obtain Type $30, (5,6)$ in the table.
\item If none of the reference points lie in the smooth plane sextic curve $C$, then the possibilities are
\[C: X^6+Y^6+Z^6+X^4L_{2,X}+X^3L_{3,X}+X^2L_{4,X}+XL_{5,X}+L_{6,X},\]
where $L_{1,X}$ does not appear since $ab\neq0.$ Clearly, $6a=6b=0\,(mod\,m),$ therefore $m=6.$
We can take $(a,b)=(5,1),$ and $(5,2).$ Thus, we obtain only (up to isomorphism) the two types  $6, (5,1)$ and $6, (5,2)$.\\
This completes the proof.
\end{itemize}

\end{proof}
%%%%%%%%%%%%%%%%%%%%%%%%%%%%%%%%%%%%%%%%%%%%%
%%%%%%%%%%%%%%%%%%%%%%%%%%%%%%%%%%%%%%%%%%%%%
\noindent{\fontsize{13.5}{9}\textbf{Concluding remarks.}}\\
We conclude the present paper with some remarks and comments.
\begin{itemize}
\item It should be noted that, the equation $F(X;Y;Z)$ that we attach
to some concrete type in Theorem \ref{thm20} can have another type for some special values of the
parameters. For example, Type $15; (10,9)$ with $\alpha=0$ has Type $30, (5,6)$;
another example is Type $3, (0,1)$ with $L_{3,Z}=0$ has Type $6, (0,1).$
\item The main theorem constitute a motivation to solve more general problems. One of these problems is the
determination of the non-trivial groups $G$ that appear as the  full automorphism group of a non-hyperelliptic genus $10$
curve over an algebraic closed field of characteristic zero. However, this problem will be the object of a forthcoming work.\\
Another problem is using the results done in \cite {B1} to investigate the geometry of higher order and multiple Weierstrass points of a certain families of genus $10$ curves with parameters.
 \end{itemize}

\bigskip\noindent


\begin{thebibliography}{9}                                                                                                %

\bibitem {Ac1}Accola, R., \emph{On the number of automorphisms of a closed Riemann surface,}
Trans. Amer. Math. Soc. \textbf{131} (1968), 398-408.

\bibitem {pa20}Badr, E. and Saleem, M., \emph{Classification of 1-Weierstrass points on Kuribayashi quartics, I (with two parameters),} International Conference on Mathematics Trends and Development (ICMTD12), arXiv: 1212.2708 [math.AG], Submitted.

\bibitem {B2}Badr, E. and Saleem, M., \emph{Classification of 1-Weierstrass points on Kuribayashi quartics, II (with three
    parameters),} The first international conference "New horizons in basic and applied science" (ICNHBAS), arXiv: 1304.2565 [math.AG], Submitted.

\bibitem {B1} Badr, E. and Saleem, M., \emph{Gap sequences of 1-Weierstrass points on non-hyperelliptic curves of genus $10$},
arXiv: 1307.0078 [math.AG], Submitted.

\bibitem {pa23}Bars, F., \emph{On the automorphisms groups of genus 3 curves}, Surveys in Mathematics and Mathematical Science,
    \textbf{2}, 83-124 (2002).

\bibitem {C2}Breuer, T., \emph{Characters and automorphism groups of compact Riemann surfaces}, London Math. Soc.
  Lect. Notes \textbf{280}, Cambridge Univ. Press (2000).

\bibitem {C1}Dolgachev, I., \emph{Topics in Classical Algebraic Geometry, I}, Private Lecture Notes in: http://www.math.lsa.umich.edu/»idolga/.

\bibitem {FK}Farkas, M., and Kra, I., \emph{Riemann Surfaces,} Springer-Verlag, 1992.

 \bibitem{H4}Hürwitz, A., \emph{Uber algebraische Gebilde mit eindeutigen Trasformationen in sich}, Math.Ann. \textbf{41}, 391-430(1893)
\bibitem{Ho}Homma, M. \emph{Automorphisms of prime order of curves}, Manuscripta Math. \textbf{33} (1980/81), no. 1, 99-109.

\bibitem {C4}Klein, F., \emph{\"{U}ber die Transformationen siebenter Ordnung der elliptischen Functionen,} Math. Annalen \textbf{14},
(1879), 428-471.



\bibitem {Ku}Kulkarni, R. S., \emph{Riemann surfaces admitting large automorphism groups. Extremal Riemann surfaces}
(San Francisco, CA, 1995), 63-79, Contemp. Math. \textbf{201}, Amer. Math. Soc., Providence, RI, 1997.




\bibitem {Mc1}Maclachlan, C., \emph{Abelian groups of automorphisms of compact Riemann surfaces,} Proc. London Math. Soc. (3) \textbf{15} (1965), 699-712.


\bibitem {M1}Magaard, K., Shaska T., Shpectorov S. and Voelklein H., \emph{The locus of curves with prescribed automorphism group},
     RIMS, Kyoto Technical Report Series, Communications on Arithmetic Fundamential Groups and Galois Theory,  arXiv:math/0205314 [math.AG] (2002).

\bibitem {Mt}Matsuno, T., \emph{Compact Riemann surfaces with large automorphism groups,} J. Math. Soc. Japan \textbf{51}
(1999), no. 2, 309-329.

\bibitem {Ri} Ries, J., \emph{Subvarieties of moduli space determined by finite groups acting on surfaces,} Trans. Amer. Math.
Soc. \textbf{335} (1993), no. 1, 385–406.

\bibitem {Si3}Singerman, D., \emph{Symmetries of Riemann surfaces with large automorphism group,} Math. Ann. \textbf{210} (1974), 17-32.

\bibitem {C3}The GAP Group, GAP–Groups, \emph{Algorithms, and Programming}, Version 4.2; 2000.
        (http://www.gap-system.org)


\bibitem {W1}Wiman, A.,  \emph{\"{U}ber die hyperelliptischen Curven und diejenigen vom Geschlechte $p=3$, welche eindeutige
 Transformationen in sich zulassen,} Bihang Kongl. Svenska Vetenskaps-Akademiens Handlingar (1895), \textbf{21 (1)}, 1–23.

\bibitem {W2}Wiman, A., \emph{\"{U}ber die hyperelliptischen Curven vom den Geschlechte $p=4,5,$ und $6$, welche eindeutige
 Transformationen in sich besitzen,} Bihang Kongl. Svenska Vetenskaps-Akademiens Handlingar (1895), \textbf{21 (3)}, 1–41.



\end{thebibliography}
\end{document}